\documentclass[a4paper,12pt]{article}
\usepackage[top=3cm, bottom=4cm, left=3cm, right=3cm]{geometry}

\usepackage[colorlinks=true,citecolor=black,linkcolor=black,urlcolor=blue]{hyperref}

\usepackage[utf8]{inputenc}
\usepackage[dvipsnames]{xcolor}
\usepackage{pgf,tikz}
\usetikzlibrary{arrows,arrows.meta,patterns,backgrounds}
\usepackage{amsmath,amsthm,amssymb}
\usepackage{mathtools,stmaryrd}
\usepackage{lmodern}
\usepackage[normalem]{ulem}
\usepackage{titling}
\usepackage[font=small,labelfont=bf]{caption}

\usepackage{type1cm,fix-cm}
\usetikzlibrary{positioning,intersections,backgrounds,decorations.markings,arrows.meta}

\theoremstyle{plain}
\newtheorem{thm}{Theorem}
\newtheorem{lemma}[thm]{Lemma}

\newtheorem{claim}[thm]{Claim}
\newtheorem{coroll}[thm]{Corollary}

\theoremstyle{definition}
\newtheorem{defn}[thm]{Definition}
\newtheorem{conject}[thm]{Conjecture}

\theoremstyle{remark}

\newcommand*\cd[1]{\tikz[baseline=(char.base)]{
            \node[shape=circle,draw,inner sep=1pt] (char) {#1};}}

\renewcommand{\to}{\rightarrow}
\newcommand{\duarr}{\,\rotatebox[origin=c]{90}{\ensuremath{\leftrightarrows}}\!\,}
\newcommand{\udarr}{\,\rotatebox[origin=c]{90}{\ensuremath{\rightleftarrows}}\!\,}
\newcommand{\uuarr}{\ensuremath{\upuparrows}}

\newcommand{\cH}{\mathcal{H}}
\newcommand{\cC}{\mathcal{C}}
\newcommand{\st}{^\star}

\newcommand{\tot}[1]{\overset{\!#1}{\!\!\longrightarrow}\!\!}
\newcommand{\backt}[1]{\overset{\!#1}{\!\!\longleftarrow}\!\!}
\newcommand{\xlonghy}{\tikz[baseline=-0.5ex]{\draw (0,0) -- (0.5,0);}}
\newcommand{\unt}[1]{\overset{\!#1}{\,\xlonghy\,}}

\begin{document}

\title{On the number of touching pairs \\ in a set of planar curves}

\author{P\'eter Gy\"orgyi\thanks{Institute for Computer Science and Control, Budapest, Hungary and \mbox{Department} of Operations Research, E\"otv\"os Lor\'and University, Budapest, Hungary. \mbox{Email: \texttt{gyorgyi.peter@sztaki.mta.hu}}}
  \and
  B\'alint Hujter\thanks{MTA-ELTE Egerv\'ary Research Group, Department of Operations Research, E\"otv\"os Lor\'and University, Budapest, Hungary. Supported by the Hungarian Scientic Research Fund grant no.~K109240. 
  Email: \texttt{hujterb@cs.elte.hu}} 
  \and
  S\'andor Kisfaludi-Bak\thanks{Department of Mathematics and Computer Science, TU Eindhoven, The Netherlands. \mbox{Supported} by NWO grant no.~024.002.003. Email: \texttt{s.kisfaludi.bak@tue.nl}}
}


\thanksmarkseries{arabic}

\maketitle

\begin{abstract}
Given a set of planar curves (Jordan arcs), each pair of which meets -- either crosses or touches -- exactly once, we establish an upper bound on the number of touchings. We show that such a curve family has $O(t^2n)$ touchings, where $t$ is the number of faces in the curve arrangement that contains at least one endpoint of one of the curves. Our method relies on finding special subsets of curves called quasi-grids in curve families; this gives some structural insight into curve families with a high number of touchings.
\end{abstract}

\medskip

\textbf{Keywords:}
Combinatorial geometry, Touching curves, Pseudo-segments

\tikzstyle{nb} = [circle, fill=black, minimum size=3pt, inner sep=0pt]
\tikzset{->-/.style={decoration={
  markings,
  mark=at position .5 with {\arrow{>}}},postaction={decorate}}}
\tikzset{-<-/.style={decoration={
  markings,
  mark=at position .5 with {\arrow{<}}},postaction={decorate}}}

\section{Introduction}

The combinatorial examination of incidences in the plane has proven to be a fruitful area of research. The first seminal results are the crossing lemma that  establishes a lower bound on the number of edge crossings in a planar drawing of a graph (Ajtai et al., Leighton \cite{crosslem,Leighton83}), and the theorem by Szemer\'edi and Trotter \cite{szemtrot83}, concerning the number of incidences between lines and points. Soon, the incidences of more general geometric objects (segments, circles, algebraic curves, pseudo-circles, Jordan arcs, etc.) became the center of attention \cite{kedlipa, tamaki98, aronov2002, agarwal04, marcus06, grounded11}.
With the addition of curves, the distinction between touchings and crossings is in order. 

Usually, the curves are either Jordan arcs, i.e., the image of an injective continuous function $\varphi:\left[0,1\right] \to \mathbb{R}^2$, or closed Jordan curves, where $\varphi$ is injective on $\left[0,1\right)$ and $\varphi(0)=\varphi(1)$. Generally, it is supposed that the curves intersect in a finite number of points, and that the curves are in general position: three curves cannot meet at one point, and (in case of Jordan arcs) an endpoint of a curve does not lie on any other curve. (For technical purposes, we will allow curve endpoints to coincide in some proofs.)

Let $P$ be a point where curve $a$ and $b$ meet. Take a circle $\gamma$ with center $P$ and a small enough radius so that it intersects both $a$ and $b$ twice, and the disk determined by $\gamma$ is disjoint  from all the other curves , and contains no other intersections of $a$ and $b$. Label the intersection points of $\gamma$ and the two curves with the name of the curve. We say that $a$ and $b$ \textit{cross} in $P$ if the cyclical permutation of labels around $\gamma$ is $abab$, and $a$ and $b$ \textit{touch} in $P$ if the cyclical permutation of labels is $aabb$.
In a family of curves, let $X$ be the set of crossings and $T$ be the set of touchings. 

The Richter-Thomassen conjecture \cite{rito} states that given a collection of $n$ pairwise intersecting closed Jordan curves in general position in the plane, the number of crossings is at least $(1-o(1))n^2$. A proof of the Richter-Thomassen conjecture has recently been published by Pach et al.~\cite{paruta}. They show that the same result holds for Jordan arcs as well.

It would be preferable to get more accurate bounds for the ratio of touchings and crossings. Fox et al.~constructed a family of x-monotone curves with ratio $|X|/|T|=O(\log n)$ \cite{fox2010}. If we restrict the number of intersections between any two curves, then it is conjectured that the above ratio is much higher. It has been shown that a family of intersecting pseudo-circles (i.e., a set of closed Jordan-curves, any two of which intersect exactly once or twice) has at most $O(n)$ touchings \cite{agarwal04}. We would like to examine a similar statement for Jordan arcs.

A family of Jordan arcs in which any pair of curves intersect at most once (apart from the endpoints) will be called a \textit{family of pseudo-segments}. Our starting point is this conjecture of 
J\'anos Pach \cite{PachPrivateCommunication}:

\begin{conject}\label{conj}
Let $\mathcal{C}$ be a family of pseudo-segments. Suppose that any pair of curves in $\mathcal{C}$ intersect exactly once. Then the number of touchings in $\mathcal{C}$ is $O(n)$.
\end{conject}

A family of pseudo-segments is \emph{intersecting} if every pair of curves intersects (i.e., either touches or crosses) exactly once outside their endpoints. 

Two important special cases of the above are the cases of \emph{grounded} and \emph{double-grounded} curves. (The definitions are taken verbatim from \cite{grounded11}.)
A collection $\cC$ of curves is \emph{grounded} if there is a closed Jordan curve $g$  called \emph{ground} such that each curve in $\cC$ has one endpoint on $g$ and the rest of the curve is in the exterior of $g$. The collection is \emph{double grounded} if there are disjoint closed Jordan curves $g_1$ and $g_2$ such that each curve $c \in \cC$ has one endpoint on $g_1$ and the other endpoint on $g_2$, and the rest of $c$ is disjoint from both $g_1$ and $g_2$. 

According to our knowledge the best upper bound is $O(n\log n)$ for the
number of touchings in a double-grounded $x$-monotone family of pseudo-segments
\cite{pach91} and we do not know any (non-trivial) result for the grounded case.

\subsection{Our contribution}
{
Let $\cC$ be an intersecting family of pseudo-segments. There is a planar graph drawing that corresponds to this family: the vertices are the crossings and touchings, and the edges are the sections of the curves between neighboring intersections. (Notice that the sections between curve endpoints and the neighboring intersections are not represented in this graph.) Consider the faces of this planar graph drawing. Let $t_{\cC}$ be the number of faces that contain an endpoint of at least one curve in $\cC$. Our main theorem can be stated as follows:

\begin{thm}\label{thm:main}
Let $\cC$ be an $n$-element intersecting family of pseudo-segments on the Euclidean plane. Then the number of touchings between the curves is $f(n) = O(t_{\cC}^2n)$.
\end{thm}

If $t_{\cC}$ is constant, this theorem settles Conjecture~\ref{conj}. 
Note that this includes the case when $\cC$ is a double-grounded intersecting family of pseudo-segments:
\begin{coroll}\label{cor:doublegrounded}
Let $\cC$ be an $n$-element double-grounded intersecting family of pseudo-segments. Then the number of touchings between the curves is $O(n)$.
\end{coroll}

A careful look at the proof of the main theorem yields the following result for grounded intersecting families of pseudo-segments:

\begin{thm}\label{thm:grounded}
Let $\cC$ be an $n$-element grounded intersecting family of pseudo-segments. Then the number of touchings between the curves is $O(t_{\cC}n)$.
\end{thm}

The intuition behind our approach can be described as follows. Curves starting in the same face of an arrangement can be thought of as curves having the same endpoints. A curve going from point $A$ to $B$ that touches some other curve $g$ can do that touching only in a constant number of ways, depending on which side of $g$ is touched and in which direction. We observe that a collection of curves going from $A$ to $B$ must therefore contain a subcollection that touch $g$ the same way, and these curves must have a very special grid-like structure, which we call \emph{quasi-grids}. 

It turns out that quasi-grids always emerge when we take two grid families of pseudo-segments, one containing curves from $A$ to $B$, the other containing curves from $C$ to $D$. Note that a curve touching all curves in a large quasi-grid has to lie outside the ``grid cells'', since it cannot cross the quasi-grid curves, and within a ``grid cell'' it could only reach at most four curves. If we find two curves touching the same large quasi-grid, then (intuitively) those two curves would have many intersections -- this is not possible in an intersecting family of pseudo-segments. We show that the number of touchings between a pair of fixed endpoint curve families is linear in the size of these families. We then use this observation to get the bound on the total number of touchings.

\section{Proof of the main theorem}

The rigorous proof of our main theorem is based upon a key lemma. Its proof anticipates and uses several technical lemmas which are detailed in Sections \mbox{\ref{sec:quasigrid} and \ref{sec:mainprf}}.

Before stating the key lemma, we introduce some notations. The notation $g \asymp h$ means that curves $g$ and $h$ touch each other. If $A$ and $B$ are (not necessarily distinct) points on the plane, then $\cC(A,B)$ denotes the set of directed curves going from $A$ to $B$. Note that here we consider curves as directed ones for technical reasons (for example, we can refer to the sides of a directed curve as \emph{left} and \emph{right}).

\begin{lemma}\label{lem:main}
Let $A,B,C,D$ be not necessarily distinct points on the plane, and $\cC_1$ and $\cC_2$ be finite disjoint curve families from $\cC(A,B)$ and $\cC(C,D)$, respectively. If $\cC_1 \cup \cC_2$ is an intersecting family of pseudo-segments, then
\begin{enumerate}
\item the number of $c_1 \asymp c_2$ touchings where $c_1 \in \cC_1$ and $c_2 \in \cC_2$ is $O(|\cC_1 \cup \cC_2|)$;
\item the number of touchings between curves of $\cC_i$ is $O(|\cC_i|)\;(i=1,2)$.
\end{enumerate}
\end{lemma}

\begin{proof}
We only consider the first claim, the second can be proven with the same tools. Suppose for contradiction that there are $\omega(|\cC_1 \cup \cC_2|)$ instances of $c_1 \asymp c_2$ touchings. 

Let $K$ be a large constant. Without loss of generality, we can suppose that each curve of $\cC_i$ touches at least $K$ curves of $\cC_j$. To see this, consider first the bipartite graph $G$ with vertex set $\cC_1 \cup \cC_2$, where the edges correspond to the $c_1 \asymp c_2$ touchings ($c_1 \in \cC_1$ and $c_2 \in \cC_2$). If $G$ has vertices of degree less than $K$, then delete those vertices and the incident edges. Iterate this procedure until the minimum degree is at least $K$ or the graph is empty. If $G$ had at least $K|\cC_1\cup\cC_2|$ edges, then this procedure cannot result in an empty graph.

Let $g\in \cC_1$ be an arbitrary curve. By Lemma \ref{lem:48quasigrids}, there is a \emph{quasi-grid} with respect to $g$ formed by at least $K/48 > 3$ curves. A quasi-grid is depicted in Figure~\ref{fig:qgrid}, the precise definition is given in Definition~\ref{defn:quasi-grid}.

Consider an ``inner'' curve $h$ in this quasi-grid. By Lemma~\ref{lem:touch2lem}, if a curve touches $h$, then it must also touch $g$ or a neighboring curve of $h$ in the quasi-grid. 
By our starting assumption, at least $K$ curves touch $h$. Then by Lemma~\ref{lem:48quasigrids}, at least $K/48$ of the curves touching $h$ must also touch another specific curve $h'$, and at least $K/(48)^2$ of these form a quasi-grid $\mathcal{Q}$ with respect to both $h$ and $h'$.

Therefore, by choosing $K \geq 4 \cdot 48^2 + 1$, the quasi-grid $\mathcal{Q}$ can be forced to contain at least five curves. This is a contradiction by Lemma~\ref{lem:quasigrid_has_unique_g}. 
\end{proof}

Next we show how Lemma~\ref{lem:main} implies Theorem~\ref{thm:main}. Let $t=t_{\cC}$.

\begin{proof}[Proof of Theorem~\ref{thm:main}]

Consider the planar graph drawing that corresponds to $\cC$. Let the faces of this planar graph drawing that contain an endpoint of at least one curve in $\cC$ be: $F_1,F_2,\dots,F_t$.

For $i=1,2,\dots,t$, let $P_i$ be an arbitrary point in the interior of $F_i$ not incident to any curve in $\cC$. Each curve endpoint inside $F_i$ can be connected to $P_i$ without adding any  intersections between the curves of $\cC$ with the exception of $P_i$.
Let $\cC'$ be the family of pseudo-segments obtained from $\cC$ by this procedure.

Partition $\cC'$ to disjoint subsets $\cC_1,\cC_2,\dots,\cC_s$ so that two curves are in the same subset if and only if their endpoints are the same. Note that $s \leq \binom{t+1}{2}$. Fix the orientation of each curve in $\cC'$ from $P_i$ to $P_j$ if $i < j$, and arbitrarily if $i=j$. 

Let $f_k$ denote the number of touchings inside $\cC_k$ and $f_{k,l}$ denote the number of touchings between $\cC_k$ and $\cC_l$. Then the total number of touchings in $\cC'$ is
\begin{align*}
  f(n) &= \sum_k f_k + \sum_{k < l} f_{k,l}
    = \sum_k O(|\cC_k|) + \sum_{k < l} O(|\cC_k| + |\cC_l|) \\
    &= O(n) + \sum_k (s-1) O(|\cC_k|) = O \left( s n \right) = O \left( t^2 n \right),
\end{align*}
where the second equation follows from Lemma \ref{lem:main}. 
\end{proof}

Notice that in case of a grounded intersecting family of pseudo-segments, we have $s=t+1$, so $O(sn)=O(tn)$, which proves Theorem~\ref{thm:grounded}.

\section{Quasi-grids and their occurrence}\label{sec:quasigrid}

\subsection{Notations and definitions}

We introduce several notations used in the paper. Let $g$ and $h$ be a pair of directed curves. If $g$ touches the left side of $h$, and they have the same direction at the touching point, then write $g \uuarr h$. (More precisely, let $\gamma$ be a circle around the intersection $P$ with a small enough radius so that it intersects both $a$ and $b$ twice, and the disk determined by $\gamma$ is disjoint  from all the other curves, and contains no other intersections of $a$ and $b$. We label the points where $g$ and $h$ enters $\gamma$ by $g$ and $h$, and assign the labels $g'$ and $h'$ to the points where they exit. We say that the right side of $g$ touches the left side of $h$ in $P$ if the counter-clockwise cyclic order of labels on $\gamma$ is $ghh'g'$.) Notice that this relation is not symmetric, i.e., $g \uuarr h \not\Leftrightarrow h \uuarr g$. If $g$ and $h$ have different directions at the touching point (so the counter-clockwise cyclic order of labels on $\gamma$ is $gg'hh'$ or $gh'hg'$), then write $g \duarr h$ or $g \udarr h$ depending on which side of $h$ is touched by $g$.} We say that $c_1$ and $c_2$ are $g$\emph{-touch equivalent} if they touch $g$ on the same side and in the same direction, i.e., $(g\uuarr c_1 \wedge g\uuarr c_2)$ or $(g\duarr c_1 \wedge g\duarr c_2)$ or $(c_1\uuarr g \wedge c_2\uuarr g)$ or $(c_1\udarr g \wedge c_2 \udarr g)$. A set of curves is $g$-touch equivalent if its elements are pairwise $g$-touch-equivalent.

For a directed curve $g$ with points $A$ and $B$ that lie on the curve in this order, let $A \tot{g} B$ be the closed directed subcurve from $A$ to $B$, and $B \backt{g} A$ will denote the reverse directed subcurve from $B$ to $A$. This notation can be iterated, e.g. if $P \in h \cap g$, then $A \tot{g} P \backt{h} Q$ denotes the curve which starts from $A \in g$, continues on $g$ to the intersection point $P$, then changes to $h$,  and goes on $h$ in reverse direction until it ends in $Q \in h$.  When referring to undirected subcurves, we use $A \unt{g} B$. Sometimes  these notations are also used to denote the ordering of points on a particular curve.

As already defined, $\cC(A,B)$ is the set of directed curves going from $A$ to $B$. For a curve $c \in \cC(A,B)$, let $c\st=c\setminus\{A,B\}$. For a set of curves $\mathcal{C}=\{c_1,c_2, \dots c_k\}$, let $\cC\st=\{c_1\st,c_2\st, \dots c_k\st\}$.

The objects called quasi-grids are the main tool of this paper. Intuitively, the below definition says that the incidences of a quasi-grid are exactly as shown in Figure~\ref{fig:qgrid}, with the exception of the points $X,Y,A$ and $B$ -- we allow these to coincide arbitrarily.

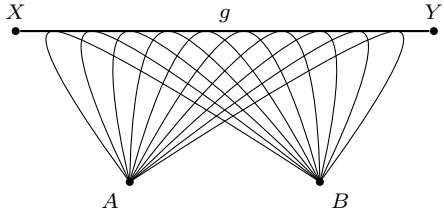
\begin{figure}
\begin{center}
\begin{tikzpicture}[x=0.5cm,y=0.4cm]
\begin{scriptsize}
\node (X) [nb, label={above:$X$}] at (0,0) {};
\node (Y) [nb, label={above:$Y$}] at (11,0) {};
\coordinate (P1) at (1,0) {};
\coordinate (P2) at (2,0) {};
\coordinate (P3) at (3,0) {};
\coordinate (P4) at (4,0) {};
\coordinate (P5) at (5,0) {};
\coordinate (P6) at (6,0) {};
\coordinate (P7) at (7,0) {};
\coordinate (P8) at (8,0) {};
\coordinate (P9) at (9,0) {};
\coordinate (P10) at (10,0) {};
\draw[thick] (X) -- node[above] {$g$}(Y);
\node (A) [nb, label={below left:$A$}] at (3,-5) {};
\node (B) [nb, label={below right:$B$}] at (8,-5) {};
\draw (A.center) .. controls ++(0,0) and ++(-1,0) .. (P1)
          .. controls ++(1,0) and ++(0,0) ..  (B.center);
\draw  (A.center) .. controls ++(0,0) and ++(-1,0) .. (P2)
          .. controls ++(1,0) and ++(0,0) ..  (B.center);
\draw  (A.center) .. controls ++(0,0) and ++(-1,0) .. (P3)
          .. controls ++(1,0) and ++(0,0) ..  (B.center);
\draw  (A.center) .. controls ++(0,0) and ++(-1,0) .. (P4)
          .. controls ++(1,0) and ++(0,0) ..  (B.center);
\draw (A.center) .. controls ++(0,0) and ++(-1,0) .. (P5)
          .. controls ++(1,0) and ++(0,0) ..  (B.center);
\draw  (A.center) .. controls ++(0,0) and ++(-1,0) .. (P6)
          .. controls ++(1,0) and ++(0,0) ..  (B.center);
\draw  (A.center) .. controls ++(0,0) and ++(-1,0) .. (P7)
          .. controls ++(1,0) and ++(0,0) ..  (B.center);
\draw  (A.center) .. controls ++(0,0) and ++(-1,0) .. (P8)
          .. controls ++(1,0) and ++(0,0) ..  (B.center);
\draw  (A.center) .. controls ++(0,0) and ++(-1,0) .. (P9)
          .. controls ++(1,0) and ++(0,0) ..  (B.center);
\draw  (A.center) .. controls ++(0,0) and ++(-1,0) .. (P10)
          .. controls ++(1,0) and ++(0,0) ..  (B.center);
\end{scriptsize}
\end{tikzpicture}
\end{center}
\caption{A quasi-grid for the case $g\upuparrows c_i$. Swapping $X,Y$ or $A,B$ gives the other 3 cases.}\label{fig:qgrid}
\end{figure}

\begin{defn}[Quasi-grid]\label{defn:quasi-grid}
A set of curves $\cC=\{c_1,c_2,\dots,c_k\}\subseteq \cC(A,B)$ forms a \textit{quasi-grid} with respect to a curve $g \in \cC(X,Y)$ if:
\begin{enumerate}
\item $\cC\st \cup \{g\st\}$ is an intersecting family of pseudo-segments
\item $\cC\st$ is $g$-touch-equivalent with touching points $P_i = g \cap c_i$ 
\item $P_{i,j} = c_i\st \cap c_j\st$ is a crossing point
\item the ordering of points on $g$ is
$P_1 \unt{g} P_2 \unt{g} \dots \unt{g} P_k$
\item the ordering of points on $c_j$ ($j=1,2,\ldots,k$) is
\[A \tot{c_j} P_{1,j} \tot{c_j} P_{2,j} \tot{c_j} \dots \tot{c_j} P_{j-1,j} \tot{c_j} P_j  \tot{c_j} P_{j,j+1}  \tot{c_j} \dots \tot{c_j} P_{j,k} \tot{c_j} B.\]
\end{enumerate}
\end{defn}

An example for a quasi-grid can be seen in Figure~\ref{fig:qgrid}. Throughout the paper (if we do not indicate it otherwise) we assume that the indices of the curves in $\cC$ describe the order of their touching points on $g$.

\subsection{Finding quasi-grids in curve configurations}

The goal of this subsection is to prove that in a family of pseudo-segments, the set of $g$-touch equivalent curves with given endpoints form a constant number of quasi-grids. 
Intuitively, Lemma~\ref{lem:S1} shows that in a family of pseudo-segments, the $g$-touch equivalent curves from $\cC(A,B)$  (where $A\neq B$) can still have two distinct types. Note that these types cannot be defined separately, only in relation to each other. In Lemma~\ref{lem:2quasigrids}, we establish that the curves in each type form a quasi-grid with respect to $g$. Lemma \ref{lem:12quasigrids} examines the case $A=B$.

\begin{lemma}\label{lem:S1}
Fix a curve $g \in \cC(X,Y)$ and suppose that $c_1$ and $c_2$ are $g$-touch-equivalent curves from $\cC(A,B)$ with touching points $P_1$ and $P_2$ respectively, where $A \neq B$. Note that $P_1$ and $P_2$ divide $c_1$ and $c_2$ into their first and second parts. Suppose further that $\{c_1\st,c_2\st,g\st\}$ is a family of pseudo-segments. Then $c_1\st$ crosses $c_2\st$ at a point $Q$, which is either the intersection of the first part of $c_1$ with the second part of $c_2$, or vice versa: the intersection of the second part of $c_1$ with the first part of $c_2$.
\end{lemma}

\begin{proof}
Suppose (without loss of generality) that $g \uuarr c_1$, $g\uuarr c_2$, and that $P_1$ precedes $P_2$ on $g$. Consider the closed directed curve $\ell = A \tot{c_1} P_1 \tot{g} P_2 \backt {c_2} A$ (curves with gray halo in the middle and right part of Figure~\ref{fig:alaplem}). We show that $\ell$ is a Jordan-curve. Suppose for contradiction that $A \tot{c_1} P_1$ and $A \tot{c_2} P_2$ has an intersection point $H \neq A$ (see the left picture in Figure~\ref{fig:alaplem}). Since $\{c_1\st,c_2\st,g\st\}$ is a family of pseudo-segments, there can be no further intersections between $c_1\st$ and $c_2\st$. It follows that $\ell'= H \tot{c_1} P_1 \tot{g} P_2 \backt {c_2} H$ is a Jordan curve that separates the plane into its left and right (shaded) side regions. Notice that $P_1 \tot{c_1} B$ begins in the right side region of $\ell'$, while $P_2 \tot{c_2} B$ begins in the left side region by our assumptions $g \uuarr c_1$ and $g \uuarr c_2$. Since $P_1 \tot{c_1} B \backt{c_2} P_2$ is a continuous curve that begins and ends in different sides of $\ell'$, it must cross $\ell'$ in a point distinct from $H$; we arrived at a contradiction.

\begin{figure}
\begin{center}
\begin{tikzpicture}[x=0.5cm,y=0.3cm]
\begin{scriptsize}
\fill [gray!50] (1,-3).. controls (2,-1) and (1.6,0) ..  (2.45,0) -- (4.45,0) .. controls (2.45,0) and ++(1,0) ..(1,-3);
\node (X) [nb, label={above:$X$}] at (0,0) {};
\node (Y) [nb, label={above:$Y$}] at (6,0) {};
\node (P1) [nb, label={above:$P_1$}] at (2.45,0) {};
\node (P2) [nb, label={above:$P_2$}] at (4.45,0) {};
\node (A) [nb, label={below left:$A$}] at (0,-5) {};
\node (H) [nb, label={below right:$H$}] at (1,-3) {}; 
\draw[thick] (X) -- node[above] {$g$} (P1) -- 
(P2) -- (Y);
\draw [blue] (A) .. controls (H) and (A) .. (H)
          .. controls (2,-1) and (1.6,0) .. node[left] {$c_1$} (P1)
          .. controls (3,0) and (3,-0.3) .. (3,-0.3);
\draw [red] (A) .. controls ++(0,1) and ++(-1,0) .. (H)
          .. controls ++(1,0) and (P1) .. node[below right] {$c_2$} (P2)
          .. controls ++(0.5,0) and ++(0,0) .. ++(0.5,-0.3);
\end{scriptsize}
\end{tikzpicture}
\begin{tikzpicture}[x=0.5cm,y=0.3cm]
\begin{scriptsize}
\node (X) [nb, label={above:$X$}] at (0,0) {};
\node (Y) [nb, label={above:$Y$}] at (6,0) {};
\node (P1) [nb, label={above:$P_1$}] at (2,0) {};
\node (P2) [nb, label={above:$P_2$}] at (4,0) {};
\node (A) [nb, label={below left:$A$}] at (0,-5) {};
\node (B) [nb, label={below right:$B$}] at (6,-5) {};
\draw[thick] (X) --node[above] {$g$} (P1) -- (Y);
\draw [name path=c1](A.center) .. controls ++(0,0) and ++(-1,0) ..node[left] {$c_1$} (P1)
          .. controls ++(1,0) and ++(0,0) ..  (B.center);
\draw [name path=c2] (A.center) .. controls ++(0,0) and ++(-1,0) ..node[above right=2pt and -5pt] {$c_2$} (P2)
          .. controls ++(1,0) and ++(0,0) ..  (B.center);
\begin{scope}[on background layer, halo/.style={gray!60, line width=1.5mm}]
\draw [halo] (A.center).. controls ++(0,0) and ++(-1,0) ..  (P1.center) -- (P2.center) .. controls ++(-1,0) and ++(0,0) ..(A.center);
\end{scope}
\node [gray,draw=none,fill=none] at (2.4,-2.2) {$\ell$};
\node [nb, label={[label distance = 0pt]below:$Q$}, name intersections={of=c1 and c2}] at (intersection-2) {};
\end{scriptsize}
\end{tikzpicture}
\begin{tikzpicture}[x=0.5cm,y=0.3cm]
\begin{scriptsize}
\node (X) [nb, label={above:$X$}] at (0,0) {};
\node (Y) [nb, label={above:$Y$}] at (6,0) {};
\node (P1) [nb, label={above:$P_1$}] at (1.5,0) {};
\node (P3) [nb, label={above:$P_2$}] at (5,0) {};
\node (A) [nb, label={below left:$A$}] at (0,-5) {};
\node (B) [nb, label={below left:$B$}] at (3,-2) {};
\draw[thick] (X) --node[above] {$g$} (P1) -- (Y);
\draw [name path=c1](A) .. controls ++(0,0) and ++(-1,0) ..node[left=-2pt,pos=0.35] {$c_1$} (P1)
          .. controls ++(1,0) and ++(0,0) ..  (B);
\draw [name path=c2] (A) .. controls ++(6,0) and ++(-1,0) .. (P3.center)  
            .. controls ++(1,0) and ++(0,-3) .. node[below] {$c_2$} (7,0) 
            .. controls ++(0,3) and ++(0,3) .. (-1,0) 
            .. controls ++(0,-4) and (B.center) .. (B.center);
\begin{scope}[on background layer, halo/.style={gray!60, line width=1.5mm}]
\draw [halo] (A.center).. controls ++(0,0) and ++(-1,0) ..  (P1.center) -- (P3.center) .. controls ++(-1,0) and ++(6,0) ..(A.center);
\end{scope}
\node [gray,draw=none,fill=none] at (3.3,-3.4) {$\ell$};
\node [nb, label={[label distance = -2pt]above left:$Q$}, name intersections={of=c1 and c2}] at (intersection-1) {};

\end{scriptsize}
\end{tikzpicture}
\caption{Left: $c_1$ and $c_2$ cannot intersect before reaching $g$; middle and right: the possible configurations}\label{fig:alaplem}
\end{center}
\end{figure}
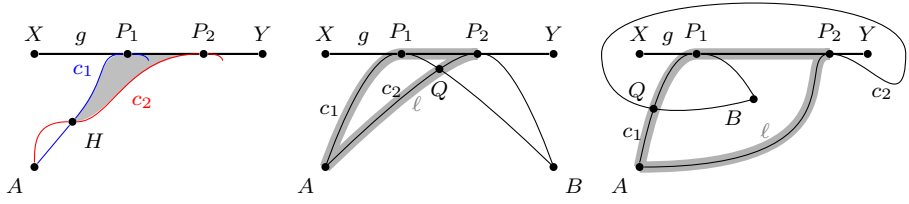

Thus $(A \tot{c_1} P_1) \cap (A \tot{c_2} P_2) = \{A\}$, hence $\ell$ is a Jordan-curve. By similar argument as above, $P_1 \tot{c_1} B \backt{c_2} P_2$ is a continuous curve that begins and ends in different sides of $\ell$, so it must cross $\ell$. Since $c_1$,  $c_2$ and $g$ are not self-intersecting and $P_1$ and $P_2$ already account for the intersections between $g$ and $\{c_1 \cup c_2\}$, the only remaining possibilities are that the crossing point is as claimed, i.e., $Q=(A \tot{c_2} P_2) \cap (P_1 \tot{c_1} B)$ or $Q=(A \tot{c_1} P_1) \cap (P_2 \tot{c_2} B)$ (see the middle and the right picture in Figure \ref{fig:alaplem}).
\end{proof}

Notice that the above lemma states that the curve $c_2$ meets $c_1$ before it meets $g$ if and only if the first part of $c_2$ crosses the second part of $c_1$ and vice versa: $c_2$ meets $g$ before it meets $c_1$ if and only if the second part of $c_2$ crosses the first part of $c_1$. This equivalence will be used several times in the following lemmas.

\begin{lemma}\label{lem:2quasigrids}
Let $g\in \cC(X,Y)$, and let $\cH$ be a set of $g$-touch-equivalent curves from $\cC(A,B)$, where $A \neq B$. If $\cH\st \cup \{g\st\}$ is a family of pseudo-segments, then $\cH$ is the disjoint union of at most two quasi-grids with respect to $g$.
\end{lemma}
\begin{proof}
We deal with the case $g \upuparrows h$ for all $h\in \cH$, the other three cases are similar.
Let $h\in \cH$ be the curve that has the first touching point on $X\tot{g}Y$ among the curves from $\cH$. Let $\cH_1\subseteq \cH$ consist of $h$ and the curves from $\cH$ that meet $h$ before they meet $g$. Let $\cH_2:=\cH\setminus\cH_1$. We prove that $\cH_1$ and $\cH_2$ are both quasi-grids with respect to $g$.

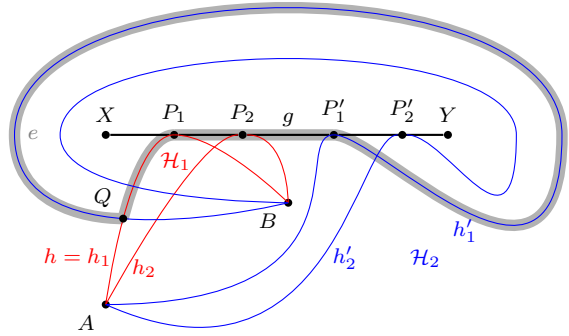
\begin{figure}
\begin{center}
\begin{tikzpicture}[x=0.6cm,y=0.45cm]
\begin{scriptsize}
\clip (-2.1,-5.8) rectangle (10.1,3.9);
\node (X) [nb, label={above:$X$}] at (0,0) {};
\node (Y) [nb, label={above:$Y$}] at (7.5,0) {};
\node (P1) [nb, label={above:$P_1$}] at (1.5,0) {};
\node (P2) [nb, label={above:$P_2$}] at (3,0) {};
\node (P3) [nb, label={above:$P'_1$}] at (5,0) {};
\node (P4) [nb, label={above:$P'_2$}] at (6.5,0) {};
\node (A) [nb, label={below left:$A$}] at (0,-5) {};
\node (B) [nb, label={below left:$B$}] at (4,-2) {};
\draw[thick] (X) -- (P1) --(P2)--node[above]{$g$} (P3) -- (Y);
\draw [name path = AP1,red] (A.center) .. controls ++(0,0) and ++(-1,0) ..node[left=-2pt,pos=0.35] {$h=h_1$} node[right=3pt, pos=0.75]{$\mathcal{H}_1$} (P1.center);
\draw[red] (P1.center) .. controls ++(1,0) and ++(0,0) ..  (B.center);
\draw [red] (A.center) .. controls ++(0,0) and ++(-1,0) ..node[right=-1pt,pos=0.3] {$h_2$} (P2)
          .. controls ++(1,0) and ++(0,0) ..  (B.center);
\draw[blue] (A.center) .. controls ++(6,0) and ++(-1,0) .. (P3.center);
\draw[blue] [name path = P3B] (P3.center)  
            .. controls ++(1,0) and ++(0,-6) .. node[below] {$h'_1$} (10,0) 
            .. controls ++(0,5) and ++(0,5) .. (-2,0) 
            .. controls ++(0,-4) and (B.center) .. (B.center);
\draw[blue]  (A.center) .. controls ++(5,-3) and ++(-1,0) .. node[right=1pt] {$h'_2$} node[right=30pt]{$\mathcal{H}_2$} (P4)  
            .. controls ++(1,0) and ++(0,-4) ..  (9,0) 
            .. controls ++(0,3) and ++(0,3) .. (-1,0) 
            .. controls ++(0,-2) and (B.center) .. (B.center);
\node(Q) [nb, label={[label distance = 0pt]above left:$Q$}, name intersections={of=AP1 and P3B}] at (intersection-1) {};
\begin{scope}[on background layer, halo/.style={gray!60, line width=1.5mm}]
\draw [halo] (Q.center).. controls ++(0,0) and ++(-0.7,0) .. (P1.center);
\draw [halo] (P1.center) --(P3.center);
\draw [halo] (P3.center)  
            .. controls ++(1,0) and ++(0,-6) .. (10,0) 
            .. controls ++(0,5) and ++(0,5) ..  (-2,0) node[right,gray]{$e$}
            .. controls ++(0,-2) and ++(-0.9,0) .. (Q.center);
\end{scope}
\end{scriptsize}
\end{tikzpicture}
\caption{Two quasi-grids with respect to $g$.}\label{fig:2grids}
\end{center}
\end{figure}

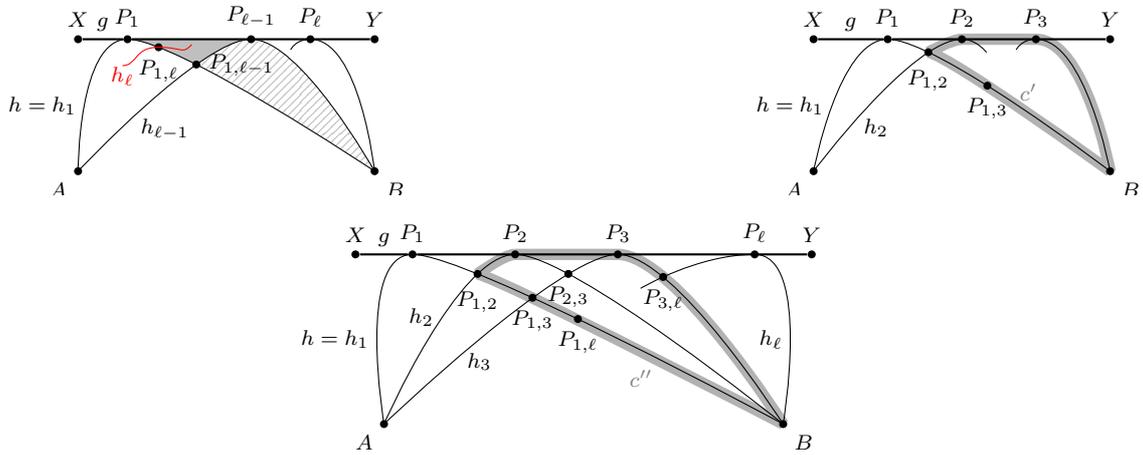
\begin{figure}
\begin{center}
\begin{tikzpicture}[x=0.65cm,y=0.35cm]
\pgfdeclarelayer{background}
\begin{scriptsize}
\clip (-1.6,-5.9) rectangle (6.6,3.9);
\node (X) [nb, label={above:$X$}] at (0,0) {};
\node (Y) [nb, label={above:$Y$}] at (6,0) {};
\node (P1) [nb, label={above:$P_1$}] at (1,0) {};
\node (P2) [nb, label={above:$P_{\ell-1}$}] at (3.5,0) {};
\node (P3) [nb, label={above:$P_{\ell}$}] at (4.7,0) {};
\node (A) [nb, label={below left:$A$}] at (0,-5) {};
\node (B) [nb, label={below right:$B$}] at (6,-5) {};
\draw[thick] (X) -- node[above] {$g$}(P1) -- (Y);

\draw (A.center) .. controls ++(0,0) and ++(-1,0) ..node[left] {$h=h_1$} (P1.center);
\draw [name path = P1B](P1.center) .. controls ++(1,0) and ++(0,0) .. node (P13) [nb, pos=0.15, label={below:$P_{1,\ell}$}] {} (B.center);
\draw [name path = AP2] (A.center) .. controls ++(0,0) and ++(-1,0) ..node[below right=2pt and -5pt] {$h_{\ell-1}$} (P2.center);
\draw (P2.center) .. controls ++(1,0) and ++(0,0) .. (B.center);
\draw  (P3.center) .. controls ++(1,0) and ++(0,0) ..  (B.center);
\draw  (P3.center) .. controls ++(-0.3,0) and ++(0,0) ..  ++(-0.4,-0.4);
\node (P12) [nb, label={right:$P_{1,\ell-1}$}, name intersections={of=P1B and AP2}] at (intersection-1) {};
\draw [red] (0.9,-1) .. controls ++(0.3,0) and ++(-0.5,0) .. (P13.center)
                     .. controls ++(0.5,0) and ++(-0.2,-0.4) .. (2.3,-0.2);
 \node [red, draw=none, fill=none] at (0.9,-1.4) {$h_{\ell}$};

\begin{scope}[on background layer]
  \draw [draw=none, pattern=north east lines, pattern color=gray!50]
  (P1.center) .. controls ++(1,0) and ++(0,0) .. (B.center) 
            .. controls ++(0,0) and ++(1,0) .. (P2.center) -- (P1.center);
  \clip (A.center) .. controls ++(0,0) and ++(-1,0) .. (P2.center) -- (X.center) -- (A.center);
  \fill [gray!50] (P1.center) .. controls ++(1,0) and ++(0,0) .. (B.center) -- (Y.center) --   (P1.center);
\end{scope}
\end{scriptsize}
\end{tikzpicture}
\hfill
\begin{tikzpicture}[x=0.65cm,y=0.35cm]
\pgfdeclarelayer{background}
\begin{scriptsize}
\clip (-1.2,-5.9) rectangle (6.6,3.9);
\node (X) [nb, label={above:$X$}] at (0,0) {};
\node (Y) [nb, label={above:$Y$}] at (6,0) {};
\node (P1) [nb, label={above:$P_1$}] at (1.5,0) {};
\node (P2) [nb, label={above:$P_2$}] at (3,0) {};
\node (P3) [nb, label={above:$P_3$}] at (4.5,0) {};
\node (A) [nb, label={below left:$A$}] at (0,-5) {};
\node (B) [nb, label={below right:$B$}] at (6,-5) {};
\draw[thick] (X) -- node[above] {$g$}(P1) -- (Y);
\draw (A.center) .. controls ++(0,0) and ++(-1,0) ..node[left] {$h=h_1$} (P1.center);
\draw [name path = P1B](P1.center) .. controls ++(1,0) and ++(0,0) .. node (P13) [nb, pos=0.4, label={below:$P_{1,3}$}] {} (B.center);
\draw [name path = AP2] (A.center) .. controls ++(0,0) and ++(-1,0) ..node[below right=2pt and -5pt] {$h_2$} (P2.center);
\draw (P2.center) .. controls ++(0.3,0) and ++(0,0) ..  ++(0.5,-.5);
\draw  (P3.center) .. controls ++(1,0) and ++(0,0) ..  (B.center);
\draw  (P3.center) .. controls ++(-0.3,0) and ++(0,0) ..  ++(-0.4,-0.4);
\node (P12) [nb, label={[label distance = 2pt]below:$P_{1,2}$}, name intersections={of=P1B and AP2}] at (intersection-1) {};
\begin{scope}[on background layer, halo/.style={gray!60, line width=1.5mm}]
\draw [halo] (P2.center) -- (P3.center)  .. controls ++(1,0) and ++(0,0) .. (B.center);
\draw [halo](P12.center) .. controls ++(0.5,-0.2) and ++(0,0) ..node[above,gray]{$c'$} (B.center);
\draw [halo](P12.center) .. controls ++(0,0) and ++(-0.2,0) .. (P2.center);
\end{scope}
\end{scriptsize}
\end{tikzpicture}

\vspace{-1cm}
\begin{tikzpicture}[x=0.75cm,y=0.45cm]
\pgfdeclarelayer{background}
\begin{scriptsize}
\clip (-2,-5.8) rectangle (7.5,3.9);
\node (X) [nb, label={above:$X$}] at (-1,0) {};
\node (Y) [nb, label={above:$Y$}] at (7,0) {};
\node (P1) [nb, label={above:$P_1$}] at (0,0) {};
\node (P2) [nb, label={above:$P_2$}] at (1.8,0) {};
\node (P3) [nb, label={above:$P_3$}] at (3.6,0) {};
\node (Pl) [nb, label={above:$P_{\ell}$}] at (6,0) {};
\node (A) [nb, label={below left:$A$}] at (-0.5,-5) {};
\node (B) [nb, label={below right:$B$}] at (6.5,-5) {};
\draw[thick] (X) --node[above] {$g$} (P1) -- (Y);
\draw [name path = AP3] (A.center) .. controls ++(0,0) and ++(-1,0) ..node[below=2pt] {$h_3$} (P3.center);
\draw [name path = P3B] (P3.center) .. controls ++(1,0) and ++(0,0) .. (B.center);
\draw (A.center) .. controls ++(0,0) and ++(-1,0) ..node[left] {$h=h_1$} (P1.center);
\draw [name path = P1B] (P1.center) .. controls ++(1,0) and ++(0,0) .. node (P1l) [nb, pos=0.42, label={below:$P_{1,\ell}$}] {} (B.center);
\draw [name path = AP2] (A.center) .. controls ++(0,0) and ++(-1,0) ..node[above left=2pt and -5pt] {$h_2$} (P2.center);
\draw [name path = P2B] (P2.center) .. controls ++(1,0) and ++(0,0) ..  (B.center);
\draw  (Pl.center) .. controls ++(1,0) and ++(0,0) .. node[left=0pt] {$h_{\ell}$} (B.center);
\draw [name path = PlA] (Pl.center) .. controls ++(-1,0) and ++(0,0) ..  ++(-2,-1);
\node (P12) [nb, label={[label distance = 2pt]below:$P_{1,2}$}, name intersections={of=P1B and AP2}] at (intersection-1) {};
\node [nb, label={[label distance = 0pt]below:$P_{2,3}$}, name intersections={of=AP3 and P2B}] at (intersection-1) {};
\node [nb, label={[label distance = 0pt]below:$P_{1,3}$}, name intersections={of=P1B and AP3}] at (intersection-1) {};
\node [nb, label={[label distance = 0pt]below:$P_{3,\ell}$}, name intersections={of=P3B and PlA}] at (intersection-1) {};
\begin{scope}[on background layer, halo/.style={gray!60, line width=1.5mm}]
\draw [halo](P2.center) -- (P3.center) .. controls ++(1,0) and ++(0,0) .. (B.center);
\draw [halo](P12.center).. controls ++(0.5,-0.2) and ++(0,0) .. node[below=2pt,gray]{$c''$} (B.center);
\draw [halo](P12.center).. controls ++(0,0) and ++(-0.2,0) ..(P2.center);
\end{scope}
\end{scriptsize}
\end{tikzpicture}
\caption{Top left: if $P_{1,\ell}$ is on $P_1 \tot{h_1} P_{1,\ell-1}$; top right: the case $|\cH_1|=\ell = 3$; bottom: the case $\ell \ge 4$.}\label{fig:sked}
\end{center}
\end{figure}

Let $\cH_1=\{h=h_1,h_2,\ldots,h_\ell\}$ and let $P_i=g\cap h_i$. Assume without loss of generality that $P_1 \tot{g} P_2 \tot{g} \dots \tot{g} P_\ell$. We show that $\cH_1$ is a quasi-grid with respect to $g$. (See Figure~\ref{fig:2grids}.)

The proof is by induction on the number of curves in $\cH_1$, for $\ell=1$ the statement is trivial. Lemma \ref{lem:S1} yields the statement for $\ell=2$.

We claim that $h_\ell$ crosses $h_1$ between $P_{1,\ell-1}$ and $B$.  By the definition of $\cH_1$ and Lemma \ref{lem:S1}, $P_{1,\ell}$ must lie on $P_1\tot{h_1} B$. Suppose for contradiction that the ordering on $h_1$ is $P_1 \tot{h_1} P_{1,\ell} \tot{h_1} P_{1,\ell-1}$ (top left of Figure~\ref{fig:sked}).
Consider the closed Jordan curve $c=P_1\tot{g}P_{\ell-1}\backt{h_{\ell-1}}P_{1,\ell-1}\backt{h_1}P_1$. (The right side region of $c$ is shaded.)

Notice that $A$ and $B$ are on the left side of $c$. To see this, consider that $c$ is made up of three curve segments, and there can be no further intersections among these three curves, so $h_1$ and $c \setminus (P_1 \tot{h_1} P_{1,\ell-1})$ are disjoint. Since the type of touching at $P_1$ is $g \uuarr h_1$, we can see that $(A \tot{h_1} P_1)\setminus {P_1}$ and consequently point $A$ in particular lies in the left side region of $c$. A similar argument for $h_{\ell-1}$ shows that $B$ is also in the left side region.

Since $h_{\ell}$ is already crossing $c$ once at $P_{1,\ell}$, it has to cross it one more time, because its endpoints $A$ and $B$ are on the same side of $c$. Since $h_{\ell}$ touches $g$ outside $c$ and it already has an intersection with $h_1$, it will cross $P_{\ell-1}\backt{h_{\ell-1}}P_{1,\ell-1}$. Now $h_{\ell}$ has entered the right side of $P_{1,\ell-1} \tot{h_{\ell-1}} B \backt{h_1} P_{1,\ell-1}$ (the region with the line pattern), while $P_{\ell}$ is on the left side (since $g\uuarr h_{\ell-1}$). So $h_{\ell}$ would have to cross $h_1$ or $h_{\ell-1}$ one more time, which is a contradiction.

If $\ell=3$, then consider the curve $c'=P_2\tot{g}P_3\tot{h_3}B\backt{h_1}P_{1,3}\backt{h_1}P_{1,2}\tot{h_2}P_2$ (see the top right of Figure~\ref{fig:sked}). Since $h_2$ and $h_3$ touch $g$ in the same direction, $h_3$ goes from $P_{1,3}$ to $P_3$ in the same side of $c'$ where $h_2$ goes from $P_2$ to $B$ (or a point on $(P_3\tot{h_3}B)$). Thus, $h_2$ and $h_3$ cross each other at a point $P_{2,3}=(P_{1,3}\tot{h_3}P_3) \cap (P_2\tot{h_2}B)$. By induction, the points on $h_1$ and $h_2$ are also in the required order.

For $\ell\ge 4$ the induction is used both for $h_1,h_2,\dots,h_{\ell-1}$ and $h_1,h_3,h_4,\dots,h_{\ell}$. We only need to show that $h_2$ and $h_\ell$ cross each other at a point $P_{2,\ell}$ which satisfies our ordering conditions. Indeed, on the right side of the curve
\[c''= P_2\tot{g}P_3\tot{h_3}P_{3,\ell}\tot{h_3}B\backt{h_1}P_{1,\ell}\backt{h_1}P_{1,2}\tot{h_2}P_2,\]
there is a crossing $P_{2,\ell}=(P_{1,\ell}\tot{h_\ell}P_{3,\ell}) \cap (P_2\tot{h_2}B)$: this shows that the ordering on $h_\ell$ is correct (see the bottom of Figure \ref{fig:sked}). A similar argument shows that the ordering of points on $h_2$ is correct, one needs to consider the right side of the following closed curve:
\[P_{1,\ell-1}\tot{h_{\ell-1}}P_{2,\ell-1}\tot{h_{\ell-1}}P_{\ell-1,\ell}\tot{h_{\ell-1}}B\backt{h_1}P_{1,\ell}\backt{h_1}P_{1,\ell-1}.\]

We show that $\cH_2$ behaves similarly. Let $\cH_2=\{h'_1,h'_2,\ldots,h'_m\}$ and let $P'_i=g\cap h'_i$. Again, suppose that $P'_1 \tot{g} P'_2 \tot{g} \dots \tot{g} P'_m$ (see Figure~\ref{fig:2grids}). By Lemma~\ref{lem:S1}, $h'_1$ must cross $h=h_1$ at a point $Q \in A\tot{h_1}P_1$, since $h'_1 \not\in \cH_1$.
Now consider the Jordan curve $e=P_1\tot{g}P'_1\tot{h'_1}Q\tot{h_1}P_1$. Again, it is easy to check that $e$ separates $A$ from $P'_j$ for $j\ge 2$. Consider a curve $h'_j \; (j\ge 2)$. It cannot meet $h_1$ before meeting $g$ since  $h'_j \not\in \cH_1$. Thus $A\tot{h'_j}P'_j$ must cross $e$ somewhere on $P'_1\tot{h'_1}Q$. We have reduced this  problem to the previous situation with $h'_1$ acting as $h_1$, so $\cH_2$ also forms a quasi-grid with respect to $g$.
\end{proof}

 In the proof of the next lemma, we use the \emph{touching graph} of a curve family of pseudo-segments. Let $\cH$ be a family of pseudo-segments. The \emph{touching graph} of $\cH$ is $G_{\cH} = (V,E)$ with $V = \cH$ and $E = \{ (a,b): a \asymp b \}$. The statement of this lemma is almost identical to the previous one, but considers the case when the endpoints of the quasi-grid coincide. In this case, we cannot prove that the curves are the union of at most two quasi-grids, but we can still bound the number of quasi-grids by a constant.

\begin{lemma}\label{lem:12quasigrids}
Let $g \in \cC(X,Y)$, and $\cH$ is a set of $g$-touch-equivalent curves from $\cC(A,A)$. If $\cH\st \cup \{g\st\}$ is an intersecting family of pseudo-segments, then $\cH$ is the disjoint union of at most 12 quasi-grids with respect to $g$.
\end{lemma}
\begin{proof}
Again, we only deal with the case $g \upuparrows h$ for all $h \in \cH$.
The first claim is that any $h \in \cH$ touches at most 2 other curves in $\cH$.
Since $h$ is a Jordan curve, it separates the plane into two regions, one of these contains $g$; denote this region by $R_g$, and the other by $R_n$ (see Figure~\ref{fig:12qgrid}, $R_n$ has a line pattern). Observe that no curve in $\cH$ touching $h$ can enter $R_n$ as such a curve cannot touch $g$. Let $P = g \cap h$. We prove that there is at most one curve in $\cH$ that touches $A\tot{h}P$. Suppose for contradiction that curves $h_1,h_2 \in \cH$ are both touching $A \tot{h} P$ at points $T_1$ and $T_2$ respectively, with the ordering $A \tot{h} T_1 \tot{h} T_2 \tot{h} P$. Let $Q_i = h_i \cap g$. Note that $A$ lies on the left side of the curve $c=Q_1 \tot{g} P \backt{h} T_1 \unt{h_1} Q_1$ since $g \uuarr h$. By an earlier observation, $h_2$ is disjoint from the open region $R_n$, which is the right side of $h$ --- so $h_2$ touches the left side of $h$ in $T_2$, i.e., $h_2\uuarr h$ or $h_2 \duarr h$. It follows that both $A \tot{h_2} T_2$ and $T_2 \tot{h_2} A$ must cross $c$, but this crossing can only happen along $T_1 \unt{h_1} Q_1$ because the other boundary curves $g$ and $h$ are touched by $h_2$. This means that $h_2\st$ and $h\st$ cross at least twice, which contradicts the basic properties of an intersecting family of pseudo-segments. A similar argument shows that there is at most one curve in $\cH$ that touches $P\tot{h}A$.

\begin{figure}
\begin{center}
\begin{tikzpicture}[x=0.8cm,y=0.4cm]
\begin{small}
\fill [pattern=north east lines, pattern color=gray!50] (3,-2.5) .. controls ++(0,1) and ++(-0.5,0) .. node [nb, label={left:$T_2$}] {} (4.5,0)
          .. controls ++(1,0) and ++(1,0) ..node[right] {$h$} (3,-5).. controls ++(1,1) and ++(0,-1) .. (3,-2.5);
\node (X) [nb, label={above:$X$}] at (0,0) {};
\node (Y) [nb, label={above:$Y$}] at (6,0) {};
\node (Q) [nb, label={above:$Q_1$}] at (1.5,0) {};
\node (P) [nb, label={above:$P$}] at (4.5,0) {};
\node (A) [nb, label={below left:$A$}] at (3,-5) {};
\node (T) [nb, label={left:$T_1$}] at (3,-2.5) {};
\node at (4,-2) {$R_n$}; 
\draw[thick] [->-] (X) --node[above]{$g$} (Q) --  (P);
\draw[thick] (P) -- (Y);
\draw [->-](A.center) .. controls ++(1,1) and ++(0,-1) .. (T.center);
\draw [->-](T.center) .. controls ++(0,1) and ++(-0.5,0) .. node [nb, label={left:$T_2$}] {} (P.center)
          .. controls ++(1,0) and ++(1,0) ..node[right] {$h$} (A.center);
\draw (A.center) .. controls ++(-1,1) and ++(0,-1) .. (T.center);
\draw (T.center) .. controls ++(0,1) and ++(0.5,0) ..  (Q.center)
          .. controls ++(-1,0) and ++(-1,0) ..node[left] {$h_1$} (A.center);
\begin{scope}[on background layer, halo/.style={gray!60, line width=1.5mm}]
\draw [halo](T.center) .. controls ++(0,1) and ++(0.5,0) ..(Q.center)--node[gray,above]{$c$} (P.center) .. controls ++(-0.5,0) and ++(0,1) .. (T.center);
\end{scope}
\end{small}
\end{tikzpicture}
\caption{Curve $h$ touches at most two other curves in $\cH$.}\label{fig:12qgrid}
\end{center}
\end{figure}

Consider the touching graph $G_\cH$. Our first observation implies that the maximal degree in $G_\cH$ is 2, thus by Brooks' theorem \cite{brooks}, $G_\cH$ is $3$-colorable. It is sufficient to prove that each color class is the disjoint union of at most 4 quasi-grids with respect to $g$.

Let $\cH_0 \subseteq \cH$ be a color class; consequently, it cannot contain a touching pair of curves, i.e., the curves in $\cH$ are pairwise intersecting. Let $k = |\cH_0|$. In this paragraph, an ending of a directed curve refers to one of the endings of its undirected version. Consider the cyclic order of the curve endings of $\cH_0$ around $A$: $x_1,x_2,\ldots,x_{2k}$. Each curve appears exactly twice in this sequence. Each pair of curves in $\cH_0$ crosses, hence $x_i$ and $x_{k+i}$ belong to the same curve for each $i \in \{ 1,\dots,k \}$. Therefore we may dilate $A$ to two points $A_1$ and $A_2$ such that endings $x_1,\dots,x_k$ are at $A_1$ and endings $x_{k+1},\dots,x_{2k}$ are at $A_2$. Now $\cH_0$ can be considered as a family of $A_1 \unt{} A_2$ curves, which is the union of $\cH_{12}$ containing $A_1 \tot{} A_2$ curves and $\cH_{21}$ containing $A_2 \tot{} A_1$ curves. According to Lemma \ref{lem:2quasigrids}, both $\cH_{12}$ and $\cH_{21}$ are the union of at most two quasi-grids with respect to $g$.
\end{proof}

The above Lemmas also imply the following one:

\begin{lemma}\label{lem:48quasigrids}
Let $A,B,C,D$ be not necessarily distinct points on the plane. Let $g \in \cC(A,B)$, and let $\cC_0 \subset \cC(C,D)$ be a finite curve family such that $\{g\} \cup \cC_0$ is an intersecting family of pseudo-segments with all $h \in \cC_0$ touching $g$. Then $\cC_0$ is the disjoint union of at most 48 quasi-grids with respect to $g$.
\end{lemma}
\begin{proof}
$\cC_0$ is the disjoint union of at most four $g$-touching equivalence classes ($g\uuarr h$, $g\duarr h$, $h \uuarr g$ and $h \udarr g$). By lemmas \ref{lem:2quasigrids} and \ref{lem:12quasigrids}, each such class can be decomposed into at most 12 quasi-grids.
\end{proof}

\section{Touching quasi-grids with external curves}\label{sec:mainprf}

\begin{lemma}\label{lem:touch2lem}
Let $g$ be any curve in $\cC(A,B)$ and let $\cH = \{h_1,h_2,h_3\} \subseteq \cC(C,D)$ be a quasi-grid with respect to $g$ (possibly a part of a larger quasi-grid). Suppose that for a curve $g' \in \cC(A,B)$, the set of five curves $\{g,g',h_1,h_2,h_3\}$ is an intersecting family of pseudo-segments and $g'$ touches the middle curve $h_2 \in \cH$. Then $g'$ must also touch at least one more among $\{g,h_1,h_3\}$.
\end{lemma}

\begin{proof} Suppose that $g \uuarr h_i \; (i=1,2,3)$, the other cases are similar. The definition of quasi-grids enumerates all intersections between the four curves $h_1,h_2,h_3$ and $g$. It follows that the borders of the faces in the right side of $C\tot{h_1} P_1 \tot{g} P_3 \tot{h_3} D \backt{h_1} P_{1,3} \backt{h_3} C$ are determined (see the faces marked with encircled numbers in Figure~\ref{fig:touch2lem}). Notice that some (or all) of $A,B,C$ and $D$ may coincide, so the other faces are unknown. Let $\cd{1}$ be the right side of $C \tot{h_1} P_{1,2} \backt{h_2} C$. In the same manner, we assign numbers $\cd{2}-\cd{5}$ to some other faces as well, see Figure~\ref{fig:touch2lem}.

Suppose for contradiction that $g'$ crosses $h_1$, $h_3$ and $g$. We need the following claim to proceed with our proof.

\begin{figure}
\begin{center}
\begin{tikzpicture}[x=0.52cm,y=0.45cm]
\begin{scriptsize}
\node (A) [nb, label={above:$A$}] at (-1,0) {};
\node (B) [nb, label={above:$B$}] at (7,0) {};
\node (P1) [nb, label={above:$P_1$}] at (0.5,0) {};
\node (P2) [nb, label={above:$P_2$}] at (3,0) {};
\node (P3) [nb, label={above:$P_3$}] at (5.5,0) {};
\node (C) [nb, label={below left:$C$}] at (0.5,-5) {};
\node (D) [nb, label={below right:$D$}] at (5.5,-5) {};
\draw[thick] (A) --node[above] {$g$} (P1)-- (B);
\draw [name path=h1] (C.center) .. controls ++(0,0) and ++(-1,0) ..node[left] {$h_1$} (P1.center)
          .. controls ++(1,0) and ++(0,0) ..  (D.center);
\draw [name path=h2] (C.center) .. controls ++(0,0) and ++(-1.1,0) ..node[pos=0.45,left] {$h_2$} (P2.center)
          .. controls ++(1.1,0) and ++(0,0) ..  (D.center);
\draw [name path=h3] (C.center) .. controls ++(0,0) and ++(-1,0) ..node[pos=0.35, right] {$h_3$} (P3.center)
          .. controls ++(1,0) and ++(0,0) ..  (D.center);     
\node at (0.7,-1.7) {$\cd{1}$};
\node at (5.3,-1.7) {$\cd{2}$};
\node at (2,-2.2) {$\cd{3}$};
\node at (4,-2.2) {$\cd{4}$};
\node at (3,-1) {$\cd{5}$};
\node (P12) [nb, label={[label distance = 1pt]left:$P_{1,2}$}, name intersections={of=h1 and h2}] at (intersection-2) {};
\node (P13) [nb, label={[label distance = 2pt]below:$P_{1,3}$}, name intersections={of=h1 and h3}] at (intersection-2) {};
\node (P23) [nb, label={[label distance = 2pt]right:$P_{2,3}$}, name intersections={of=h2 and h3}] at (intersection-2) {};
\begin{scope}[on background layer, halo/.style={gray!60, line width=1.5mm}]
\draw [halo] (C.center) .. controls ++(0,0) and ++(-1.1,0) .. (P2.center)
          .. controls ++(1.1,0) and ++(0,0) ..  (D.center); 
\draw [halo]  (P13.center) .. controls ++(0,0) and ++(0,0) ..node[gray,below right=0.1cm and -0.2cm]{$c$} (D.center);
\draw [halo]  (C.center) .. controls ++(0,0) and ++(0,0) .. (P13.center);
\end{scope}

\end{scriptsize}
\end{tikzpicture}
\hfill
\begin{tikzpicture}[x=0.52cm,y=0.45cm]
\begin{scriptsize}
\coordinate (A) at (-1,0) {};
\coordinate (B) at (7,0) {};
\node (P1) [nb, label={above:$P_1$}] at (0.5,0) {};
\node (P2) [nb, label={above:$P_2$}] at (3,0) {};
\node (P3) [nb, label={above:$P_3$}] at (5.5,0) {};
\node (C) [nb, label={below left:$A=B=C$}] at (0.5,-5) {};
\node (D) [nb, label={left:$D$}] at (5.5,-4) {};
\draw[thick] (C.center) .. controls ++(0,0) and ++(-2,0) .. (A) --node[above left] {$g$} (P1)-- (B)
.. controls ++(2,0) and ++(10,-1.5) .. (C.center);
\draw [name path=h1] (C.center) .. controls ++(0,0) and ++(-1,0) ..node[left] {$h_1$} (P1.center)
          .. controls ++(1,0) and ++(0,0) ..  (D.center);
\draw [name path=h2] (C.center) .. controls ++(0,0) and ++(-1.1,0) ..node[pos=0.45,left] {$h_2$} (P2.center)
          .. controls ++(1.1,0) and ++(0,0) ..  (D.center);
\draw [name path=h3] (C.center) .. controls ++(0,0) and ++(-1,0) ..node[pos=0.35, right] {$h_3$} (P3.center)
          .. controls ++(1,0) and ++(0,0) ..  (D.center);
\node at (0.7,-1.7) {$\cd{1}$};
\node at (5.3,-1.7) {$\cd{2}$};
\node at (2,-2.2) {$\cd{3}$};
\node at (4.3,-2) {$\cd{4}$};
\node at (3,-1) {$\cd{5}$};
\node (P12) [nb, label={[label distance = 1pt]left:$P_{1,2}$}, name intersections={of=h1 and h2}] at (intersection-2) {};
\node (P13) [nb, label={[label distance = 2pt]below:$P_{1,3}$}, name intersections={of=h1 and h3}] at (intersection-2) {};
\node (P23) [nb, label={[label distance = 2pt]right:$P_{2,3}$}, name intersections={of=h2 and h3}] at (intersection-2) {};
\begin{scope}[on background layer, halo/.style={gray!60, line width=1.5mm}]
\draw [halo] (C.center) .. controls ++(0,0) and ++(-1.1,0) ..(P2.center) .. controls ++(1.1,0) and ++(0,0) ..  (D.center);
\draw [halo]  (P13.center) .. controls ++(0,0) and ++(0,0) ..node[gray,below right=0.1cm and -0.2cm]{$c$} (D.center);
\draw [halo]  (C.center) .. controls ++(0,0) and ++(0,0) .. (P13.center);
\end{scope}

\end{scriptsize}
\end{tikzpicture}

\begin{tikzpicture}[x=0.5cm,y=0.45cm]
\begin{scriptsize}
\coordinate (A) at (-1,0) {};
\coordinate (B) at (7,0) {};
\node (P1) [nb, label={above:$P_1$}] at (0.5,0) {};
\node (P2) [nb, label={above:$P_2$}] at (3,0) {};
\node (P3) [nb, label={above:$P_3$}] at (5.5,0) {};
\node (C) [nb, label={below left:$A=C$}] at (0.5,-5) {};
\node (D) [nb, label={below right:$B=D$}] at (5.5,-5) {};
\draw[thick] (C.center) .. controls ++(0,0) and ++(-2,0) .. (A)--node[above left] {$g$} (P1) -- (B)
.. controls ++(2,0) and ++(0,0) .. (D.center);
\draw [name path=h1] (C.center) .. controls ++(0,0) and ++(-1,0) ..node[left] {$h_1$} (P1.center)
          .. controls ++(1,0) and ++(0,0) ..  (D.center);
\draw [name path=h2] (C.center) .. controls ++(0,0) and ++(-1.1,0) ..node[pos=0.45,left] {$h_2$} (P2.center)
          .. controls ++(1.1,0) and ++(0,0) ..  (D.center);
\draw [name path=h3] (C.center) .. controls ++(0,0) and ++(-1,0) ..node[pos=0.35, right] {$h_3$} (P3.center)
          .. controls ++(1,0) and ++(0,0) ..  (D.center);
\node at (0.7,-1.7) {$\cd{1}$};
\node at (5.3,-1.7) {$\cd{2}$};
\node at (2,-2.2) {$\cd{3}$};
\node at (4,-2.2) {$\cd{4}$};
\node at (3,-1) {$\cd{5}$};
\node (P12) [nb, label={[label distance = 1pt]left:$P_{1,2}$}, name intersections={of=h1 and h2}] at (intersection-2) {};
\node (P13) [nb, label={[label distance = 2pt]below:$P_{1,3}$}, name intersections={of=h1 and h3}] at (intersection-2) {};
\node (P23) [nb, label={[label distance = 2pt]right:$P_{2,3}$}, name intersections={of=h2 and h3}] at (intersection-2) {};
\begin{scope}[on background layer, halo/.style={gray!60, line width=1.5mm}]
\draw [halo] (C.center) .. controls ++(0,0) and ++(-1.1,0) ..(P2.center) .. controls ++(1.1,0) and ++(0,0) ..  (D.center);
\draw [halo]  (P13.center) .. controls ++(0,0) and ++(0,0) ..node[gray,below right=0.1cm and -0.2cm]{$c$} (D.center);
\draw [halo]  (C.center) .. controls ++(0,0) and ++(0,0) .. (P13.center);
\end{scope}

\begin{scope}
\fill[draw=none, pattern=dots, pattern color=gray!50] (C.center) .. controls ++(0,0) and ++(-2,0) .. (A)--(P2)  .. controls ++(-1.1,0) and ++(0,0) ..(C.center);
\end{scope}
\begin{scope}
\fill[draw=none, pattern=crosshatch dots, pattern color=gray!50] (D.center) .. controls ++(0,0) and ++(1.1,0) .. (P2) -- (B)
.. controls ++(2,0) and ++(0,0) .. (D.center);
\end{scope}
\end{scriptsize}
\end{tikzpicture}
\hspace{-0.5cm}
\begin{tikzpicture}[x=0.5cm,y=0.45cm]
\begin{scriptsize}
\coordinate (A) at (-1,0) {};
\coordinate (B) at (7,0) {};
\node (P1) [nb, label={above:$P_1$}] at (0.5,0) {};
\node (P2) [nb, label={above:$P_2$}] at (3,0) {};
\node (P3) [nb, label={above:$P_3$}] at (5.5,0) {};
\node (C) [nb, label={below left:$A=B=C=D$}] at (3,-5) {};
\node (D) at (3,-5) {};
\draw[thick] (C.center) .. controls ++(-5,-1) and ++(-2,0) .. (A)--node[above left] {$g$} (P1) -- (B)
.. controls ++(2,0) and ++(5,-1) .. (D.center);
\draw [name path=h1] (C.center) .. controls ++(-3,0) and ++(-3,0) ..node[left] {$h_1$} (P1.center)
          .. controls ++(2,0) and ++(1,2) ..  (D.center);
\draw [name path=h2] (C.center) .. controls ++(-3,1) and ++(-1.1,0) ..node[pos=0.35,left] {$h_2$} (P2.center)
          .. controls ++(1.1,0) and ++(3,1) ..  (D.center);
\draw [name path=h3] (C.center) .. controls ++(-1,2) and ++(-2,0) .. (P3.center)
          .. controls ++(3,0) and ++(3,0) .. node[pos=0.65, right] {$h_3$} (D.center);
\node at (0.7,-1.7) {$\cd{1}$};
\node at (5.3,-1.7) {$\cd{2}$};
\node at (2.3,-1.7) {$\cd{3}$};
\node at (3.7,-1.7) {$\cd{4}$};
\node at (3,-1) {$\cd{5}$};
\node (P12) [nb, label={[label distance = 1pt]left:$P_{1,2}$}, name intersections={of=h1 and h2}] at (intersection-2) {};
\node (P13) [nb, label={[label distance = 1pt]below  right:$P_{1,3}$}, name intersections={of=h1 and h3}] at (intersection-2) {};
\node (P23) [nb, label={[label distance = 2pt]right:$P_{2,3}$}, name intersections={of=h2 and h3}] at (intersection-2) {};
\begin{scope}[on background layer, halo/.style={gray!60, line width=1.5mm}]
\draw [halo] (C.center) .. controls ++(-3,1) and ++(-1.1,0) .. (P2.center)
          .. controls ++(1.1,0) and ++(3,1) ..  (D.center);
\draw [halo]  (P13.center) .. controls ++(0.4,-0.7) and ++(0.4,0.7) ..node[gray,right]{$c$} (D.center);
\draw [halo]  (C.center) .. controls ++(-0.4,0.7) and ++(-0.4,-0.7) .. (P13.center);
\end{scope}

\begin{scope}
\fill[draw=none, pattern=dots, pattern color=gray!50] (C.center) .. controls ++(-5,-1) and ++(-2,0) .. (A)--(P2)   .. controls ++(-1.1,0) and ++(-3,1) ..(C.center);
\end{scope}
\begin{scope}
\fill[draw=none, pattern=crosshatch dots, pattern color=gray!50] (C.center).. controls ++(3,1) and ++(1.1,0) .. (P2) -- (B)
.. controls ++(2,0) and ++(5,-1) .. (D.center);
\end{scope}
\end{scriptsize}
\end{tikzpicture}

\end{center}
\caption{The figures for various possible equalities among $A,B,C$ and $D$.}\label{fig:touch2lem}
\end{figure}
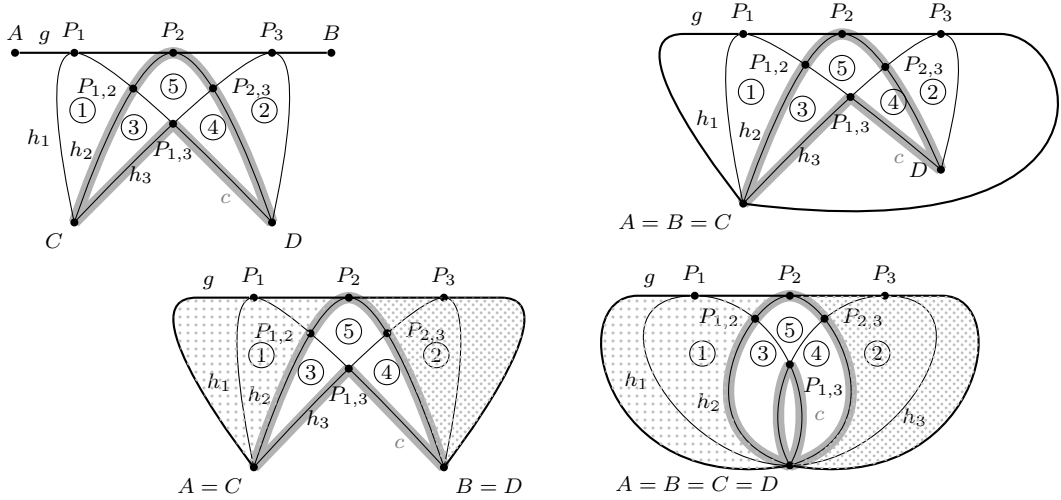

\begin{claim}\label{claim:touch2lem}
The curve $g'$ cannot enter region \emph{$\cd{5}$}.
\end{claim}

\begin{proof}
If $g'$ passes through $\cd{5}$, then -- since it touches $h_2$ --, it must cross both $P_{1,2} \tot{h_1} P_{1,3}$ and $P_{1,3} \tot{h_3} P_{2,3}$. Since there can be no more intersections with $h_1$ or $h_3$, the curve $g'$ cannot pass through the closed curve $c=C \tot{h_2} D \backt{h_1} P_{1,3} \backt{h_3} C$, therefore it cannot meet $g\st$ (see Figure \ref{fig:touch2lem}).
\end{proof}

Since $g'$ must cross $h_1$, it has to enter either region $\cd{1}$ or $\cd{4}$ (by Claim~\ref{claim:touch2lem} it cannot cross $P_{1,2}\tot{h_1}P_{1,3}$). If it enters $\cd{1}$, then --- since it has crossed $h_1$, one of its endpoints $A$ or $B$ has to be on the border of $\cd{1}$, thus either $A=C$ or $B=C$. If $g'$ enters $\cd{4}$, then by Claim~\ref{claim:touch2lem}, one of the endpoints is $D$, so $A=D$ or $B=D$.
The curve $g'$ also needs to cross $h_3$, so it enters either $\cd{2}$ or $\cd{3}$, and as before, it follows that $A=D$ or $B=D$ in case of entering $\cd{2}$ and $A=C$ or $B=C$ otherwise.

If $g'$ enters $\cd{1}$ and $\cd{3}$, then $A=C=B$. Since $\cd{1}$ and $\cd{3}$ are on the same side of the closed curve $g \in \cC(A,A)$, the curve $(g')\st \in \cC\st(A,A)$ crosses $g\st$ at an even number of points, we arrived at a contradiction. The case when $g'$ enters $\cd{2}$ and $\cd{4}$ is identical if we swap the role of $C$ and $D$.

If $g'$ enters $\cd{1}$ and $\cd{2}$, then the endpoints of $g'\in\cC(A,B)$ are $C$ and $D$. If $A=D$ and $B=C$, then the closed curves $B \tot{h_1} P_1 \tot{g} B$ and $A \tot{g} P_1 \tot{h_1} A$ must cross each other at an even number of points, so there is an intersection point distinct from $P_1$. Note that $h_1$ and $g$ are members of an intersecting family of pseudo-segments (since $\cH$ is a quasi-grid with respect to $g$), so the intersection must be at their endpoints, thus $A=B$. Consequently, if $g'$ enters $\cd{1}$ and $\cd{2}$, then either $A=B=C=D$ or $A=C$ and $B=D$ are two distinct points (see the bottom of Figure \ref{fig:touch2lem}).
Let $R_1$ be the region to the left of $A\tot{h_2} P_2 \backt{g} A$ (sparsely dotted) and $R_2$ be the left side of $B\backt{g}P_2\tot{h_2}B$ (densely dotted). Notice that $g'$ starts in $R_1$ and ends in $R_2$, two regions that are guaranteed to be disjoint apart from $P_2, A$ and $B$. Since it cannot cross $h_2$,  it crosses both $g_1\st$ and $g_2\st$, where $g_1=A \tot{g} P_2$ and $g_2=P_2 \tot{g} B$. This is a contradiction since $g'$ has to cross $g$ exactly once.

If $g'$ enters $\cd{3}$ and $\cd{4}$, then $A=C$ and $B=D$ by the same argument as in the previous case. Let $A'$ and $B'$ be points on $g'$ close to its starting and endpoint $A$ and $B$, so that there are no touchings or crossings on $g'$ between $A$ and $A'$ and between $B'$ and $B$. Note that $g'$ cannot cross $h_2$ because they need to touch. Thus the boundary of $R_1$ can only be crossed on $g_1$, and both $A'$ and $B'$ lie outside $R_1$ (they are in $\cd{3}$ and $\cd{4}$ respectively), so the number of intersections between $g_1\st$ and $(g')\st$ is even. The same argument holds for $R_2$ and $g_2$, so the number of intersections between $g\st$ and $(g')\st$ is even -- a contradiction.
\end{proof}

The next lemma demonstrates our intuitive claim that touching the members of a large quasi-grid by two curves is not possible inside an intersecting family of pseudo-segments.

\begin{lemma}\label{lem:quasigrid_has_unique_g}
Let $\cH$ be a set of at least 5 curves from $\cC(A,B)$, where $A$ and $B$ may coincide. Let $g_1,g_2$ be two curves such that $\cH \cup \{g_1,g_2\}$ form a family of pseudo-segments. Then $\cH$ cannot form a quasi-grid with respect to both $g_1$ and $g_2$. 
\end{lemma}

\begin{proof}
Suppose for contradiction that $\cH$ is a quasi-grid with respect to $g_1$ and $g_2$ simultaneously. Let $\cH'=\{ h_1,h_2, \dots, h_5 \}$ be a 5-element subset of $\cH$ that touch $g_1$ in this order at $P_1,\dots ,P_5$, see Figure~\ref{fig:quasigrid_of_5_curves}.

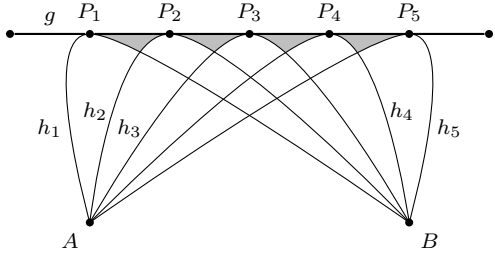
\begin{figure}
\begin{center}
\begin{tikzpicture}[x=0.7cm,y=0.5cm]
\begin{scriptsize}
\node (A) [nb] at (-1,0) {};
\node (B) [nb] at (8,0) {};
\coordinate (P1F) at (0.5,0.3) {};
\coordinate (P5F) at (6.5,0.3) {};
\node (P1) [nb, label={above:$P_1$}] at (0.5,0) {};
\node (P2) [nb, label={above:$P_2$}] at (2,0) {};
\node (P3) [nb, label={above:$P_3$}] at (3.5,0) {};
\node (P4) [nb, label={above:$P_4$}] at (5,0) {};
\node (P5) [nb, label={above:$P_5$}] at (6.5,0) {};
\node (C) [nb, label={below left:$A$}] at (0.5,-5) {};
\node (D) [nb, label={below right:$B$}] at (6.5,-5) {};
\draw[thick] (A)--node[above]{$g$} (P1) -- (B);
\draw (C.center) .. controls ++(0,0) and ++(-1,0) ..node[left]{$h_1$} (P1.center)
          .. controls ++(1,0) and ++(0,0) ..  (D.center);
\draw (C.center) .. controls ++(0,0) and ++(-1.1,0) ..node[above left =0cm and-0.1cm]{$h_2$} (P2.center)
          .. controls ++(1.1,0) and ++(0,0) ..  (D.center);
\draw (C.center) .. controls ++(0,0) and ++(-1,0) ..node[left]{$h_3$} (P3.center)
          .. controls ++(1,0) and ++(0,0) ..  (D.center);
\draw (C.center) .. controls ++(0,0) and ++(-1,0) .. (P4.center)
          .. controls ++(1,0) and ++(0,0) ..node[above right =0cm and-0.1cm]{$h_4$}  (D.center);
\draw (C.center) .. controls ++(0,0) and ++(-1,0) .. (P5.center)
          .. controls ++(1,0) and ++(0,0) ..node[right]{$h_5$}  (D.center);
\begin{scope}[on background layer]
  \clip (C.center) .. controls ++(0,0) and ++(-1,0) .. (P1.center)
          .. controls ++(1,0) and ++(0,0) .. (D.center) -- (P5F) -- (P1F) -- (C.center);
  \clip (C.center) .. controls ++(0,0) and ++(-1.1,0) .. (P2.center)
          .. controls ++(1.1,0) and ++(0,0) ..  (D.center)-- (P5F) -- (P1F) -- (C.center);
  \clip (C.center) .. controls ++(0,0) and ++(-1,0) .. (P3.center)
          .. controls ++(1,0) and ++(0,0) ..  (D.center)-- (P5F) -- (P1F) -- (C.center);
  \clip (C.center) .. controls ++(0,0) and ++(-1,0) .. (P4.center)
          .. controls ++(1,0) and ++(0,0) ..  (D.center)-- (P5F) -- (P1F) -- (C.center);
  \clip (C.center) .. controls ++(0,0) and ++(-1,0) .. (P5.center)
          .. controls ++(1,0) and ++(0,0) ..  (D.center)-- (P5F) -- (P1F) -- (C.center);
  \fill [gray!50,draw=none] (3,-5) -- (P5.center) -- (P1.center) -- (3,-5);
\end{scope}
\end{scriptsize}
\end{tikzpicture}
\end{center}
\caption{A 5-element quasi-grid $\cH'$}\label{fig:quasigrid_of_5_curves}
\end{figure}

The curve $g_2$ cannot have any points in a region which is enclosed by only curves from $\cH'$: it cannot leave the region since it cannot cross any of $\cH'$, and every region is bounded by at most four of the $\cH'$ curves, so at least one curve would remain untouchable for $g_2$.

Consequently, $g_2$ has to touch $h_2$, $h_3$ and $h_4$ in the regions enclosed by $g_1, h_i$ and $h_{i+1}\;(i=1,2,3,4)$ (see the shaded regions in Figure~\ref{fig:quasigrid_of_5_curves}). Since $g_2$ can meet $g_1$ at most once, it can visit only one of these regions, so at least one of $h_2,h_3$ and $h_4$ will remain untouchable -- we arrived at a contradiction. 
\end{proof}

\section*{Acknowledgements}
We thank J\'anos Pach and G\'eza T\'oth for suggesting the original problem, for the encouragement and for the fruitful discussions. We thank an anonymus referee for several remarks that improved the presentation of the paper.

\bibliographystyle{abbrv}
\bibliography{curves}

\end{document}